\documentclass[11pt]{amsart}
\usepackage{geometry}
\usepackage{amsmath}
\usepackage[all]{xy}
\usepackage{bbm}
\usepackage{amsthm}
\usepackage{amssymb}
\usepackage{extpfeil}
\usepackage{mathdots}
\usepackage{tikz-cd}

\usepackage[square, sort, comma, numbers]{natbib}
\usepackage{hyperref}
\usepackage[capitalize,nameinlink]{cleveref}
\usepackage{aliascnt}

\theoremstyle{plain}
\newtheorem{thm}{Theorem}[section]

\newcommand{\newsharedtheorem}[3]{%
  \newaliascnt{#1}{thm}%
  \newtheorem{#1}[#1]{#2}%
  \aliascntresetthe{#1}%
  \crefname{#1}{#2}{#3}%
}

\theoremstyle{definition}
\newsharedtheorem{defin}{Definition}{Definitions}
\newsharedtheorem{exm}{Example}{Examples}
\newsharedtheorem{rmk}{Remark}{Remarks}
\newsharedtheorem{fact}{Fact}{Facts}

\theoremstyle{plain}
\newtheorem*{theorem*}{Theorem}
\newsharedtheorem{lem}{Lemma}{Lemmas}
\newsharedtheorem{cor}{Corollary}{Corollaries}
\newsharedtheorem{prop}{Proposition}{Propositions}

\crefname{step}{Step}{Steps}

\newcommand{\zset}{\mathbbm{Z}} 
\newcommand{\rset}{\mathbbm{R}} 
\newcommand{\cset}{\mathbbm{C}} 
\newcommand{\zp}{\mathbf{Z}_{p}}  
\newcommand{\qp}{\mathbbm{Q}_{p}}  
\newcommand{\qpb}{\overline{\mathbbm{Q}}_{p}}
\newcommand{\ql}{\mathbbm{Q}_{\ell}}  
\newcommand{\zl}{\zset_{\ell}}
\newcommand{\fpb}{\overline{\mathbbm{F}}_{p}}  

\newcommand{\fq}{\mathbbm{F}_{q}}  
\newcommand{\flb}{\overline{\mathbbm{F}}_{\ell}}  

\newcommand{\ok}[2][]{#1\mathcal{O}_{#2}}

\newcommand{\ind}[4][]{\mathrm{Ind}#1_{#2}^{#3}(#4)}
\newcommand{\tens}[4][]{#1#2\otimes_{#3}#4}
\newcommand{\tensl}[4][]{#1#2\otimes^L_{#3}#4}

\newcommand{\coho}[4][]{#1 H^{#2}(#3, #4)}

\newcommand{\ccoho}[4][]{#1 H_c^{#2}(#3, #4)}

\newcommand{\gl}[3][]{#1\mathrm{GL}_{#2}({#3})}  

\newcommand{\gm}{\mathbf{G}_m}

\newcommand{\rphtc}{R\pi_{HT,*}\flb}

\newcommand{\flg}{\mathcal{F}\ell}

\title[]{On the non-generic part of cohomology of compact unitary Shimura varieties of signature $(1,n)$}
\date{\today}
\author{Kun Liu}
\address{Department of Mathematics, University of Chicago, Chicago, IL 60637}
\email{liukun801@uchicago.edu}
\begin{document}
\maketitle
\begin{abstract}
In this short note, we prove a result about the {\em non-generic} part of the cohomology of certain compact unitary Shimura varieties for good $p$, partially extending a result of Boyer in the case of Harris--Taylor unitary Shimura varieties. Our arguments are different to those of Boyer --- we work in the context of the work of Fargues--Scholze, using ideas introduced by Koshikawa to study the generic part of cohomology.
\end{abstract}
\section{Introduction}
\subsection{Background: Generic Vanishing}
In \cite{caraiani2017generic}, Caraiani and Scholze prove that, under certain generic conditions, the $\flb$ cohomology of compact unitary Shimura varieties is concentrated in the middle degree. A key input is the `perversity'(up to shift) of $\rphtc$ given by the Hodge-Tate period map $$\pi_{HT}:\operatorname{Sh}_{K^p}\to\flg.$$ This ultimately rests on the `affineness' and `partial properness' nature of $\pi_{HT}$. They also analyze the cohomology of the Igusa varieties using trace formulas.

Then, Koshikawa \cite{koshikawa2021generic} introduces a new approach using Fargues--Scholze parameters. This method is closely related to the cohomology of local Shimura varieties rather than the Igusa variety. He then proves the local vanishing for the generic part and suggests possibilities for the non-generic case. It is these possibilities that we are exploring in this note.

More recently, Zhang \cite{zhang2023pel} constructs the Igusa stack sitting in a diagram:  

$$\begin{tikzcd}
\operatorname{Sh}_{K^p}^{\circ,\diamond} \arrow[d] \arrow[r, "\pi_{HT}^{\diamond}"] &\flg^{\diamond} \arrow[d, "BL"] \\
\mathrm{Igs} \arrow[r, "\bar{\pi}_{HT}"]         & \mathrm{Bun}_{G,\mu^{-1}}   .
\end{tikzcd}$$

By base change, we have that $R\pi_{HT,*}^{\diamond}\flb=BL^*R\bar{\pi}_{HT,*}\flb$ where $BL$ is the Beauville--Laszlo map. Since the Igusa variety (essentially the fiber of $\pi_{HT}$) is affine, this establishes the genuine perversity for $R\bar{\pi}_{HT,*}^{\diamond}\flb $ on $\mathrm{Bun}_G$. 

Moreover, $BL^*$ almost agrees with the Hecke operator acting on $\mathrm{Bun}_G$ as defined in \cite{fargues2021geometrization}. Hamann--Lee \cite{hamann2023torsion} prove that after localizing at generic parameters, the Hecke operator is $t$-exact, working primarily on the $\mathrm{Bun}_G$ side. This directly implies torsion vanishing outside the middle degree. In a similar vein, Zou \cite{zou2025categorical} and Nguyen \cite{nguyen2025categorical} deal with generalized generic parameters (including ramified ones) for $\mathrm{GL_n}$, involving an analysis of categorical local Langlands (see also \cite{nguyen2023categorical} for related discussions). With recent progress on categorical local Langlands by Zhu \cite{zhu2025tame}, Yang--Zhu \cite{yang2025generic} prove torsion vanishing for unramified generic parameters for more general groups in a more categorical way.

\subsection{Non-Generic Case}
Arthur’s conjecture posits a relationship between the Weil group representation equipped with the Lefschetz-$\mathrm{SL}_2$ structure and the automorphic representation occurring in the $L^2$ cohomology of Shimura varieties. One may expect a similar version for torsion coefficients related to the non-generic case. In short, we hope there is a relationship between the $\mathrm{SL}_2$ action for the parameter corresponding to an automorphic representation and the cohomological degree where it appears.

This has been proved in the Harris--Taylor case by Boyer \cite{boyer2019torsion}, and for Hilbert modular varieties by Caraiani and Tamiozzo \cite{caraiani2023etale}. Recently, Koshikawa and Shin \cite{koshikawa2024non} formulate conjectures in this direction. They also use local-global compatibility to transfer results from Archimedean places to deal with the characteristic 0 case. 

In \cite{nguyen2025categorical}, Nguyen establishes general vanishing results for unitary Shimura varieties of signature $(1,n)$, covering non-generic parameters that are not necessarily unramified. 

\begin{rmk}
The argument of \cite{nguyen2025categorical} uses Hodge--Newton reducibility to extend the vanishing results from the generalized generic case. The present note treats the narrower setting of unramified non-generic parameters, which allows for a specialized and direct argument.
\end{rmk}

\subsection{Main Results} We reprove Boyer's result for split $p$ by an argument in the spirit of \cite{koshikawa2021generic} for compact unitary Shimura varieties $\operatorname{Sh}_K$ of signature $(1,n)$. Choosing $\qpb\cong\cset$ appropriately, let $r_{\mu}$ be the highest weight representation corresponding to $\mu$. We consider $\coho i{\operatorname{Sh}_K}{\flb}$, and we assume $\ell> n+1$
(see Sections 3–4 for the precise hypotheses):

\begin{thm}[\Cref{main}]
Assume exactly one archimedean place has signature $(1,n)$. Let $p\not= \ell$ and $p>n+1$. Assume $K_p$ is hyperspecial and $p$ satisfies additional split and unramified condition. Let $H_{K_p}$ be the Hecke algebra at $p$ and $\mathfrak{m}\subset H_{K_p}$ be the maximal ideal corresponding to an unramified L-parameter $\rho_{\mathfrak{m}}$. For $1\leq r\leq n$, if there is no $\alpha$ such that the multiset $\{\alpha,p\alpha,p^2\alpha,\dots,p^{r}\alpha\}$ is contained in the multiset of eigenvalues of $r_{\mu}\circ\rho_{\mathfrak m}(\mathrm{Frob})$, then for $i\leq n-r$ or $i\geq n+r$, $$\coho i{\operatorname{Sh}_K}{\flb}_{\mathfrak m}=0.$$
\end{thm}

Furthermore, if one allows signatures of the form $(1,n)$ to occur in several places, then a straightforward calculation gives us:
\begin{thm}[\cref{main2}]
With similar conditions as above, If there is no $\alpha$ such that the multiset $\{\alpha,p\alpha,p^2\alpha,\dots,p^{r}\alpha\}$ is contained in the multiset of the eigenvalues of $r_{\mu}\circ\rho_{\mathfrak m}(\mathrm{Frob})$, then $\coho i{\operatorname{Sh}_K}{\flb}_{\mathfrak m}=0$, for $i\leq N-r$ or $i\geq N+r$ where $N=\mathrm{dim}\ \operatorname{Sh}_K$.
\end{thm}
\subsection{Overview of proof}
In Section \ref{sec2}, we prove a representation result describing how representations coming from division algebras look like. Then, in Section \ref{sec3}, using \cite{koshikawa2021generic}, we can match the Hecke action with the action of the spectral Bernstein center, which forces local vanishing after localizing at suitable Hecke eigensystems on strata of large dimensions by representation input. Finally, in Section \ref{sec4}, Mantovan’s formula provides us with the required filtration to pass from strata to the global variety. 

\begin{rmk}
In this specific case, the degree bound of the partially compact support cohomology of the Igusa variety (which, in our case, is just the usual cohomology) by affineness and the degree shift by the $J_b^0$ (connected) torsor structure of the local Shimura variety would already yield a result.
\end{rmk}

\subsection{Further Remarks}
As a point of reference, I share my current understanding of Boyer’s idea, which has benefited greatly from explanations by Koshikawa. In \cite{boyer2019torsion}, he establishes the vanishing by analyzing the mod $p$ special fiber of the Harris–Taylor Shimura variety. He used the fact that the stratification is a parabolic induction of small space (related to Hodge--Newton reducibility in more general contexts). Then, noticing that the Lubin–-Tate space at the hyperspecial level is essentially a disk, he obtains the appearance of $Speh_r(1)$ for the cohomology of the corresponding stratum. 

Combining this geometric input with Mantovan's formula yields a proof of the main theorem that is considerably shorter (and works under weaker conditions on $\ell$ and $p$) and essentially recovers Boyer’s argument. Our approach, however, does not rely on this input and may be more flexible for general signatures. Moreover, Boyer also proves a stronger statement related to the real Galois action on the \'etale cohomology, formulated as in \cite[Theorem~1.3]{koshikawa2020vanishing}.

\section{Representation Result}\label{sec2}
First, we prove the representation result needed for our conclusion. In particular, we want to prove that Galois representations will be of the form $\oplus_{i=0}^{n-1}\chi_{cycl}^i$ (mod $\ell$) for those arising from a division algebra of degree $n$ via the Jacquet–Langlands correspondence under certain conditions.

\begin{rmk}
After completing the argument in the special case needed here, we learned that a more general result had already been proved by Vign\'eras \cite{vigneras1997propos}. We nevertheless present the argument in this restricted setting.
\end{rmk}

We say that a Weil--Deligne representation is integral if it admits an invariant lattice.

\begin{thm}\label{rep}
Let $F$ be a non-Archimedean field with residue field $\fq$ of characteristic $p\not=\ell$. Let $\rho=(\rho',N)$ be an n-dimensional, indecomposable, integral, $\ell$-adic, Frobenius-semisimple Weil--Deligne representation of $W_F$. Then the semisimplification of the modulo $\ell$ reduction of $\rho'$ is $\oplus_{i=0}^{r-1}\tens{\bar{\rho}^{ss}_0}{}{\chi_{cycl}^i}$ for some irreducible representation $\rho_0$ of $W_F$.

Assuming further that $(\bar\rho')^{ss}$ is unramified, then $\bar\rho_0^{ss}$ is also unramified.

Moreover, if $(\ell,n)=1$ and $(p,n)=1$, then up to a twist by an unramified character, $\bar{\rho}_0=\oplus_{i=0}^{d-1}\chi_{cycl}^i$ with $d$ being the smallest integer such that $\ell|q^d-1$ or $d=1$. Hence, $$\bar\rho'=\oplus_{i=0}^{dr-1}\chi_{cycl}^i.$$
\end{thm}

\begin{rmk}
We hope this result could also be a possible step towards cases of general signatures. In fact, the existence of irreducible $\rho_0$ forces $\chi_{cycl}^d=1\ (\mathrm{mod}\ \ell)$, making it easier to find long sequences of the form $\{\alpha,\alpha q,\alpha q^2,\dots,\alpha q^u\}$. Also, the requirement that $d$ is the smallest integer such that $\ell|q^d-1$ suggests that this could only happen at a fixed degree. So the discussion will mainly focus on representations purely of the form $\mathrm{Sp}(r)$. 
\end{rmk}

\begin{proof}
We may assume that the coefficients take values in $M$ for some finite extension $M$ of $\ql$, and $\Lambda$ is an invariant lattice. Since the representation is $\mathrm{Frob}$-semisimple and indecomposable, it can be written as $\tens{\rho_0}{}{}\mathrm{Sp}(r)$, where $\rho_0$ is an absolutely irreducible representation of the Weil group $W_F$ with trivial monodromy. Then $N^{r-1}\Lambda$ is a nontrivial invariant sub-lattice with trivial monodromy, making $\rho_0$ an integral Weil-Deligne representation. So we have $\bar\rho^{ss}=\oplus_{i=0}^{r-1}\tens{\bar{\rho}_0^{ss}}{}{\chi_{cycl}^i}\ (\mathrm{mod}\ \ell)$ . The rest of the theorem follows from the lemma below. 
\end{proof}
\begin{lem}
Let $\rho_0$ be an irreducible representation of the Weil group $W_F$ of dimension $d>1$ such that $\bar\rho_0^{ss}$ is unramified. If $p\not=\ell$, $(p,d)=1$, and $(\ell,d)=1$, then $\rho_0=\ind {W_L}{W_F}{\chi}$ (up to unramified twist) for the unramified extension $L/F$ of degree $d$, and character $\chi$ of $W_L$ whose modulo $\ell$ reduction  is trivial. Moreover, this could only happen when $\ell|q^d-1$ and $d$ is the smallest integer satisfying  this condition.

In particular, we see that $$\bar{\rho}_0=\oplus_{i=0}^{d-1}\chi_{cycl}^i$$ (up to unramified twists).
\end{lem}
\begin{proof}
By \cite[Cor 2.2.5.3]{tate1979number}, $\rho_0$ is monomial. So, we may assume that $\rho_0=\ind {W_L}{W_F}\chi$ with $\chi:W_L\to \ok M^{\times}$ a smooth character for some finite extension $M$ of $\ql$. Then, $\chi$ could not extend to $W_{L'}$ whenever $F\subset L'\subsetneq L$. Indeed, by Frobenius reciprocity, if $\chi$ has an extension $\chi':W_{L'}\to \ok M^{\times}$, then there exists a non-trivial map $\chi' \to\ind {W_L}{W_{L'}}{\chi}$. So, we have a non-trivial map $\ind {W_{L'}}{W_F}{\chi'}\to\ind {W_L}{W_F}{\chi}$, contradicting irreducibility.

Now we focus on $\chi$. Note that $\chi$ factors through an abelian quotient. By class field theory, we have that
$$W_L^{ab}\cong \text{pro-$p$ part}\times k_L^{\times}\times \mathrm{Frob}^{\zset}\xrightarrow{\chi} \ok M^{\times}\cong \zset_{\ell}^{u}\times \mu_{\ell^v}\times k_{M}^{\times}$$
where $k_L,k_M$ denotes the corresponding residue fields. 

We will first prove that $\chi$ must be just of the form $\chi:k_L^{\times}\to \mu_{\ell^v}$ (up to unramified twist). In particular, $\bar\chi=1$ (i.e. the projection to $k_M^{\times}$ is trivial).

(1) After possibly enlarging $M$, any unramified characters of $W_L$ could extend to $W_F$. As $\ind{H}G{\tens {\rho}{}{\chi'_{|H}}}=\tens{\ind HG{\rho}}{}{\chi'}$ for $\chi'$ a character of $G$, we may assume $\mathrm{Frob}$ maps to the trivial element. 

(2) Since $p\not=\ell$ and we consider the discrete topology on the right-hand side, the pro-$p$ part only maps non-trivially to $k_M^{\times}$. Besides, the image of $k_L^{\times}$ does not contribute to the torsion-free part of the right-hand side. 

(3) So, we are now left with the map $$\chi:\text{pro-$p$ part}\times k_L^{\times}\to\mu_{\ell^v}\times k_M^{\times}.$$ We claim that $\mathrm{Im}(\bar\chi)$ is $\ell$-torsion, so $\bar\chi=1$ and the pro-$p$ part can't contribute. 

Let us prove the claim. Since we are considering $\mathrm{mod}\ \ell$ coefficients, for $a\in\text{pro-p part}$ or $a\in k_L^{\times}$ such that $(\operatorname{ord}(a),\ell)=1$, $\bar\rho_0(a)$ is semisimple. So, $\bar\rho_0^{ss}(a)=\bar\rho_0(a)$. Since $\bar\rho_0^{ss}$ is unramified, $\bar\rho_0(a)=1$. Then, by the property of induced representation, we see that $\bar\chi(a)=1$. So, $\mathrm{Im}(\bar\chi)$ is $\ell$-torsion.

Now we prove that $L/F$ is unramified and analyze $d$ to deduce the final conclusion,

Recall for $F\subset L'\subset L$, we have the commutative diagram

$$
\begin{tikzcd}
W_{L}^{ab} \arrow[r, two heads] \arrow[d] & k_{L}^{\times} \arrow[d, "\mathrm{Nm}"'] \\
W_{L'}^{ab} \arrow[r, two heads]                    & k_{L'}^{\times}.                    
\end{tikzcd}
$$
Since we allow enlarging $M$, the existence of extension is equivalent to $\mathrm{Ker} (\mathrm{Nm})\subset \mathrm{Ker} (\chi)$.

(4) We claim that $L/F$ is unramified, which implies $|k_L^{\times}|=q^d-1$. Otherwise, we can find $L'$ such that $L/L'$ is totally ramified. Then $k_L=k_{L'}$ and $\mathrm{Nm}:x\to x^a$ for some $a|d$. In particular, we have $(a,\ell)=1$. Since $\mathrm{Im}(\chi)$ is $\ell$ torsion, $Ker(\mathrm{Nm)}\subset Ker(\chi)$, so $\chi$ could extend.

(5) Now we prove that $d$ is the minimal integer such that $\ell||k_L^{\times}|=q^d-1$. Indeed, since $d\not=1$, $\chi$ is nontrivial, so $\ell|q^d-1$. Choose the smallest $c$ such that $\ell|q^c-1$. If $c<d$, then $c|d$. So, we can choose $F\subset L'\subset L$ such that $[L:L']=h, [L':F]=c$. Since $L/L'$ is unramified, then
\begin{equation*}
    \begin{aligned}
        \textnormal{Nm:}\quad &k_L^\times\longrightarrow k_{L'}^\times\\
        &x\ \ \ \mapsto \ \  x^{1+q^c+q^{2c}+\cdots +q^{(h-1)c}}.
    \end{aligned}
\end{equation*}
As $\ell|q^c-1$, $1+q^c+q^{2c}+\dots+q^{(h-1)c}\equiv h(\mathrm{mod}\ \ell)$. Since $(h,\ell)=1$, $\chi(\mathrm{Ker}(\mathrm{Nm}))=1$. So, $\chi$ could extend and we arrive at a contradiction.

So, $q$ is the generator of the $d$-cyclic part of $k_M^{\times}$, which implies that $\bar{\rho}_0=\ind{W_L}{W_F} 1=\oplus_{i=0}^{d-1}\chi_{cycl}^i$.
\end{proof}

\section{Local Vanishing Result}\label{sec3}
With a similar strategy, we will now generalize \cite[Thm 1.1]{koshikawa2021generic} to certain non-generic cases, i.e., we prove that after localizing at certain non-generic parameters, the cohomology of the local Shimura variety vanishes depending on the form of $J_b$. 

We work in a restricted setting of \cite[Thm 1.1/1.2]{koshikawa2021generic}. Namely, we consider a local Shimura datum $(G,b,\mu)$ such that $G_{\qp}\cong \prod_{i=1}^d \gl {n+1}{\qp}\times\gm(\qp)$. We choose $\mu$ so that its restriction to the first factor of the $\prod_{i=1}^d \gl {n+1}{\qp}$ is of the form $\mu': t\to (t,1,\dots,1,1)$, trivial otherwise and is the identity on the similitude factor. $b$ is an element in $B(G,\mu^{-1})$. Choose $K_p$ to be the hyperspecial open compact subgroup. Let $\mathcal{M}_{(G,b,\mu),K_p}$ denote the corresponding local Shimura variety. Then the Hecke algebra $H_{K_p}=\zl[K_p\backslash G(\qp)/K_p]$ acts on it via correspondence. Besides, let $J_b$ be the algebraic group associated with $b$. We have $(\mathrm{Aut}_G(\widetilde{X_b})_{\eta}^{ad}=)\widetilde{J_b}=J_b^0\rtimes J_b(\qp)$ acting on $\mathcal{M}_{(G,b,\mu),K_p}$ as well. In particular, $J_b^0$ is connected and has dimension $d_b=\langle2\rho,v_b\rangle$. Note that $d_b$ will also be the dimension of the corresponding Igusa variety in the global setting.  

By projecting onto the first factor, we have that $B(G,\mu^{-1})=B(\mathrm{GL}_{n+1},\mu'^{-1})$. Recall that elements of $B(\mathrm{GL}_{n+1},\mu'^{-1})$ correspond to $p$-divisible groups over $\fpb$ of height $n+1$ and dimension $1$, which are $\{G_{\frac{1}{r}}\times({\qp/\zp})^{n+1-r}\}_{1\leq r\leq n+1}$. (We use $G_{\frac{r}{s}}$ to denote the unique irreducible $p$ divisible group of dimension $r$ and height $s$.)  We denote the corresponding elements by $b_r$. Then, we have $b_1>b_2>b_3>\dots>b_{n+1}$ and $d_{b_r}=n-r+1$.

Let $\mathfrak m\subset H_{K_p}$ be the maximal ideal corresponding to a character $\lambda_{\mathfrak m}:H_{K_p}\to\flb$. Then we have the associated unramified $L$-parameter,

$$\rho_{\mathfrak m}:\mathrm{Frob}_{\qp}^{\zset}\to \prod_{i=1}^d \gl {n+1}{\flb}\times\gm(\flb).$$

\begin{rmk}
For the local result, there is no essential difference between whether we consider $\qp$ or one of its finite extension. We work with $\qp$ for ease of notation.
\end{rmk}

Because of the choice of $\mu$, we will only consider the first factor of $\rho_{\mathfrak{m}}$ and $J_b$. 

\begin{rmk}\label{first}
For similar reason, one can also treat the more general local situation where $G_{\qp}=\textnormal{Res}_{K/\qp}\mathrm{GL}_{n+1}\times H$, with $K/\qp$ unramified, $H$ unramified over $\qp$,  and $\mu$ nontrivial only on the factor $\textnormal{Res}_{K/\qp}\gl {n+1} {K}$.  With a more careful analysis of $B(G,\mu^{-1})$ and addition conditions on $\ell$ coming from the input of Fargues--Scholze, one should get similar vanishing results relating to eigenvalues of $\mathrm{Frob}_{K}$. 
\end{rmk}

Recall that for any algebraically closed field $M$ over $\overline{\zset}_{\ell}$, \cite{fargues2021geometrization} constructs a semisimple local Langlands correspondence, which to any irreducible smooth $M$ representation of $J_b(\qp)$ associates a semisimple $L$-parameter:

$$\varphi_{\pi}:W_{\qp}\xrightarrow{\varphi_{J_b}^1(\pi)} \hat{J_b}(M)\rtimes W_{\qp}\xrightarrow{twist} \hat{G}(M)\rtimes W_{\qp}$$ well-defined up to $\hat{J_b}(M)$-conjugacy, satisfying certain conditions. This construction is compatible with parabolic induction and mod $\ell$ reduction. Moreover, for basic $b$ it is the same as the usual semisimplified parameter given by Jacquet--Langlands by \cite{hansen2022kottwitz}.
This construction depends on a choice of $p^{\frac 12}\in\overline{\zl}$, which we fix from now on.

\begin{rmk}\label{twist}
The $L$-parameter constructed above is built by considering the action of spectral Bernstein center through $j^b:\mathrm{Bun}_G^b\to \mathrm{Bun}_G$. We call it $\varphi_{G}^b(\pi)(=\varphi_{\pi})$. If we consider it purely as $J_b$ representation, and work on $\mathrm{Bun}_{J_b}^{1}\to \mathrm{Bun}_{J_b}$, then by \cite[IX.7.1]{fargues2021geometrization}, the parameter we get will be different from the above $\varphi_{G}^b(\pi)$ by a twist of unramified characters ($\rho_{\hat G}-\rho_{\hat{J_b}}$), which we denote by $\varphi_{J_b}^1(\pi)$. In particular, for $b_r$, we have that $J_{b_r}(\qp)=D^{\times}_{\frac {1}r}\times \gl{n+1-r}{\qp}$, where $D_{\frac 1 r}$ is the division algebra over $\qp$ with invariant $\frac 1r$. If we decompose $\pi$ as $\pi=(\pi_1\boxtimes\pi_2)$, then  by writing $\varphi_{J_{b_r}}^1(\pi)$ as $(\varphi_1,\varphi_2)$ with $\varphi_1=\varphi_{D^{\times}_{\frac {1}r}}^1(\pi_1)$ and $\varphi_2=\varphi_{\gl{n+1-r}{\qp}}^1(\pi_2)$,  $\varphi_{G}^b(\pi)$ will be $(\varphi_1 \chi_{cycl}^{-\frac{n-r+1}2},\varphi_2\chi_{cycl}^{\frac r2})$. Additionally, when $b$ is basic, there is no difference between the two parameters.
\end{rmk}

Then, Koshikawa proves the following essential result, matching the action of the spectral Bernstein center and the Hecke action.
\begin{thm}\cite[Section 3]{koshikawa2021generic}
If an irreducible smooth $J_b(\qp)$ representation $\pi$ appears as a subquotient of $\ccoho k{\mathcal{M}_{(G,b,\mu),K_p}}{\flb}_{\mathfrak{m}}$, then $$\rho_{\mathfrak m}=\varphi_{\pi}.$$
\end{thm}


\subsection{Proof of Local Vanishing}
We will always assume $p,\ell>n+1$ in this section.
\begin{lem}
If $\pi$ is an irreducible smooth $\flb$ representation of $J_{b_i}(\qp)$, and $\varphi_{\pi}$ is unramified, then there exists $\alpha$ such that the multiset $\{\alpha, p\alpha, p^2\alpha,p^3\alpha,\dots, p^{i-1}\alpha\}$ is contained in the multiset of eigenvalues of $\varphi_{\pi}(\mathrm{Frob})$. (We will call this property $P_i$ for convenience.)
\end{lem}
\begin{rmk}
The property $P_i$ is preserved under duality.
\end{rmk}
\begin{proof}
As in \cref{twist}, $J_{b_i}(\qp)=D_{\frac {1}i}^{\times}\times\gl {n+1-i}{\qp}$. So, $\pi$ is of the form $(\pi_1\boxtimes\pi_2)$ for irreducible representations $\pi_1,\pi_2$ of $D_{\frac {1}i}^{\times}$ and $\gl {n+1-i}{\qp}$ respectively. 
Hence, $\varphi_{\pi}$ will be of the form $(\varphi_{\pi_1}\chi_{cycl}^{-\frac{n-i+1}2},\varphi_{\pi_2}\chi_{cycl}^{\frac i2})$. Since $D^{\times}$ is compact modulo the center, $\pi_1$ is supercuspidal. So, by \cite[3.27,3.28]{minguez2014types}, we can always find a lift of  $\pi_1$ to a supercuspidal irreducible characteristic 0 representation $\tilde{\pi}$. Then, by the Jacquet-Langlands correspondence, its corresponding $L$-parameter $\phi$ is indecomposable. Now, via the compatibility of Fargues--Scholze parameter with the usual Jacquet--Langlands correspondence, its semisimplification $\phi^{ss}$ is given by $\varphi_{\tilde{\pi}}$. Then since Fargues--Scholze parameter is compatible with mod $\ell$ reduction, we see that $\overline{\phi^{ss}}$ equals $\varphi_{\pi_1}$. Then by \cref{rep}, $\varphi_{\pi_1}=\oplus_{t=0}^{i-1}\chi_{cycl}^t$.
\end{proof}

Let $r_{\mu}$ be the highest weight representation corresponding to $\mu$, which is the standard representation of the first factor. Then we can prove the following result:

\begin{cor}
If $\ccoho k{\mathcal{M}_{(G,b_i,\mu),K_p}}{\flb}_{\mathfrak{m}}\not=0$ for some $k$, then there exists some $\alpha$ such that the multiset $\{\alpha, p\alpha, p^2\alpha,p^3\alpha,\dots, p^{i-1}\alpha\}$ is contained in the multiset of eigenvalues of $r_{\mu}\circ\rho_{\mathfrak{m}}(\mathrm{Frob})$.
\end{cor}

Moreover, by a similar process but allowing $(1,n)$ to appear $s$ times and choosing $\mu$ accordingly, we have that $B(G,\mu^{-1})=\{{b_I=b_{i_1}\times b_{i_2}}\times\dots\times b_{i_s}| 1\leq i_k\leq {n+1}\ for\ 1\leq k\leq s \}$. Then, $J_{b_I}=\prod J_{b_{i_k}}$. Let $|I|=\sum_{j=1}^s i_j-s=\dim \flg^{b_I} $, then 
\begin{thm}\label{ssig}
If $\ccoho k{\mathcal{M}_{(G,b_I,\mu),K_p}}{\flb}_{\mathfrak{m}}\not=0$ for some $k$, then there exists some $\alpha$ such that multiset $\{\alpha, p\alpha, p^2\alpha,p^3\alpha,\dots, p^{|I|}\alpha\}$ is contained in the multiset of eigenvalues of $r_{\mu}\circ\rho_{\mathfrak{m}}(\mathrm{Frob})$.
\end{thm} 

\begin{proof}
If $\mu=\prod_{i=1}^s\mu_i$, then $r_{\mu}=\otimes r_{\mu_i}$. By a similar strategy as above, we see that the $t$-th factor of $\rho_{\mathfrak m}$ satisfies $P_{i_t}$. Hence, $r_{\mu}\circ \rho_{\mathfrak m}$ satisfies $P_{|I|+1}$.
\end{proof}

\section{Mantovan's formula and Conclusion}\label{sec4}

Now we fix the setting to an integral Shimura datum $(G,X)$ of PEL type A given by $(\mathcal{O}_B,*,\Lambda,(·,·))$, where $B$ is a central simple $F$ algebra for a CM field $F$, and $G$ is a unitary
similitude group attached to this data. Then $G(\rset)\cong G(\prod_{\tau:F\to \cset} U(p_{\tau},q_{\tau}))$. We require that the signatures are all of the form $(1,n)$ ($s$ times) or $(0,n+1)$ and $\mu$ to be the canonical cocharacter accordingly. We further require $G$ to be anisotropic so that our Shimura variety is compact. 

Now we choose $p>n+1$ such that it splits completely in $F$ and $G$ is unramified at $p$. Then, the local data at $p$ fit into our local vanishing setup. By choosing $\iota:\qpb\xrightarrow{\sim}\cset$ appropriately, we can make $\mu$ nontrivial for certain factors.  Let $K=K^pK_p$ be a neat, compact, open subgroup that is hyperspecial at $p$, such that $Sh_K$ has a smooth integral model over $\mathcal{O}_{E_v}$ where $E$ is the reflex field and $v$ the prime over $p$ induced by $\iota$. 

\begin{rmk}
The condition that $p$ is split is not essential as discussed in \cref{first}.
\end{rmk}

We now recall the reformulation of Mantovan's formula due to \cite{daniels2024igusa} (first noticed by Koshikawa and then generalized by Hamann-Lee), i.e., the Newton stratification provides a filtration of the cohomology as follows:
\begin{thm}\cite[Thm 8.5.7]{daniels2024igusa}
There exists a filtration on $R\Gamma(\operatorname{Sh}_{K^p},\flb)$ by complexes of smooth representations of $G(\qp)\times W_{E_v}$, whose graded pieces are given by 
\begin{align*}
&(j^1)^*T_{\mu}(Rj^{b}_{!}j^{b*}R\bar{\pi}_{HT,*}\flb
[-d])(-d/2)\\\cong &\tensl{R\Gamma(Ig^b,\flb)^{\mathrm{op}}}{C_c(J_b(\qp))}{R\Gamma_c(\mathcal{M}_{(G,b,\mu),\infty},\flb(d_b))[2d_b]}\\
\end{align*}
\end{thm}
By taking $K_p$ invariants on both sides, we get the filtration for $R\Gamma(\operatorname{Sh}_{K},\flb)$.

\begin{rmk}
This actually equals to $R\Gamma_c(\flg^b,\rphtc)$ where $\flg^b$ is the stratum that corresponds to $b$ on the flag variety defined by \cite{caraiani2017generic}. 
\end{rmk}
\begin{rmk}
Note that the above theorem is a particular case of a general formula.

For general $V\in D(\mathrm{Bun}_{G}^b,\flb)$, we have that
$$(j^1)^*T_{\mu}(Rj^{b}_{!}V)[-d](-\frac d2)=\tensl{V^{\mathrm{op}}}{C_c(J_b(\qp))}{R\Gamma_c(\mathcal{M}_{(G,b,\mu),\infty},\flb(d_b))[2d_b]}.$$ 
The formula essentially comes from the following Cartesian diagram:
$$\begin{tikzcd}
\mathcal{M}_{(G,b,\mu),\infty}\arrow[d] \arrow[r, "\pi_{HT}^{b}"] & \flg^b \arrow[d, "BL"] \\
*\arrow[r]         & \mathrm{Bun}_{G,\mu^{-1}}^b=[*/\widetilde{J_b}].   
\end{tikzcd}$$
One can see that the shift by $2d_b$ comes from the connected part $J_b^0$ of $\widetilde{J_b}$.
\end{rmk}

Now we assume that $\ell,p>n+1$.
\begin{cor}\label{bvanishi}
If $r_{\mu}\circ\rho_{\mathfrak m}$ does not satisfy $P_i$, then $R\Gamma_c(\flg^{b_i}/K_p,\rphtc)_{\mathfrak m}=0$.
\end{cor}

\begin{thm}\label{main}
In the case that only one place is of signature $(1,n)$, if there is no such $\alpha$ such that the multiset $\{\alpha,p\alpha,p^2\alpha,\dots,p^{r}\alpha\}$ is contained in the multiset of the eigenvalue of $r_{\mu}\circ\rho_{\mathfrak m}(\mathrm{Frob})$, then for $i\leq n-r$ or $i\geq n+r$,
$$\coho i{\operatorname{Sh}_K}{\flb}_{\mathfrak m}=0.$$ 
\end{thm}
\begin{proof}
Since $\rho_{\mathfrak m}$ does not satisfy $P_{r+1}$, then by \cref{bvanishi}, we only need to consider the graded piece for $b_t$ for $t\leq r$. We want to analyze the degree bound of
$$\tensl{R\Gamma(Ig^{b_t},\flb)^{\mathrm{op}}}{C_c(J_{b_t}(\qp))}{R\Gamma_c(\mathcal{M}_{(G,b_t,\mu),K_p},\flb(n+1-t))[2(n+1-t)]}.$$

Since $\mathcal{M}_{(G,b_t,\mu),K_p}$ is a smooth and partially proper rigid analytic variety over $C$ of dimension $n$ by \cite[Prop 5.4]{rapoport2014towards}, we know that $R\Gamma_c(\mathcal{M}_{(G,b_t,\mu),K_p},\flb)$ lies in $[0,2n]$. By the argument in \cite[Prop 8.6.2]{daniels2024igusa}, we know that $Ig^{b_t}$ is affine. So, we have that $R\Gamma(Ig^{b_t},\flb)$ lies in $[0,n+1-t]$. So the whole complex lies in $[0,n+t-1]\subset[0,n+r-1]$. So, we see that $\coho i{\operatorname{Sh}_K}{\flb}_{\mathfrak m}=0$ for $i\geq n+r$. Then by duality, we have the result of the other direction.
\end{proof}

For general $s$, by the same argument with the notations above \cref{ssig}, we have
\begin{thm}\label{main2}
In general, if there is no such $\alpha$ such that the multiset $\{\alpha,p\alpha,p^2\alpha,\dots,p^{r}\alpha\}$ is contained in the multiset of the eigenvalues of $r_{\mu}\circ\rho_{\mathfrak m}(Frob)$, then $\coho i{\operatorname{Sh}_K}{\flb}_{\mathfrak m}=0$, for $i\leq ns-r$ or $i\geq ns+r$.
\end{thm}

\section*{Acknowledgments}
I am very grateful to my advisor, Prof.~Matthew Emerton, for suggesting this problem, for many insightful conversations, and for reading and improving earlier drafts of this note. I especially thank Prof.~Teruhisa Koshikawa for valuable comments and for many helpful explanations and clarifications of related material. I also thank Yuting (Samanda) Zhang, Jinyue Luo, Deding Yang, and Ray Li for discussions and comments that improved the exposition.

{\small
\nocite{*}
\bibliographystyle{alpha}
\bibliography{bib/bib}
}
\end{document}